\documentclass[12pt]{amsart}
\usepackage{amsmath,srcltx,amssymb}
\usepackage{fullpage}

\theoremstyle{plain}
\newtheorem{thm}{Theorem}[section]
\newtheorem{pro}[thm]{Proposition}
\newtheorem{lem}[thm]{Lemma}
\theoremstyle{definition}
\newtheorem{defn}[thm]{Definition}
\newtheorem{exa}[thm]{Example}
\newtheorem{rem}[thm]{Remark}
\newtheorem{cor}[thm]{Corollary}
\numberwithin{equation}{section}

\newcommand{\R}{\mathbb{R}}

\newcommand{\N}{{{\Bbb N}}}

\newcommand{\T}{{{\Bbb T}}}

\def\qed{\hbox to 0pt{}\hfill$\rlap{$\sqcap$}\sqcup$\medbreak}

\begin{document}

\title[Discontinuous second--order problems]{A Schauder-type theorem for discontinuous operators with applications to second-order BVPs}
\date{}

\subjclass[2010]{Primary 47H10, secondary 34A36, 34B15}%
\keywords{Schauder's theorem, fixed point, discontinuous operator, separated boundary conditions}%

\author[R. Figueroa]{Rub\'en Figueroa}%
\address{Rub\'en Figueroa, Departamento de An\'alise Ma\-te\-m\'a\-ti\-ca, Facultade de Matem\'aticas,
Universidade de Santiago de Com\-pos\-te\-la, 15782 Santiago de Compostela, Spain}%
\email{ruben.figueroa@usc.es}%

\author[G. Infante]{Gennaro Infante}
\address{Gennaro Infante, Dipartimento di Matematica e Informatica, Universit\`{a} della
Calabria, 87036 Arcavacata di Rende, Cosenza, Italy}%
\email{gennaro.infante@unical.it}%

\begin{abstract}
We prove a new fixed point theorem of Schauder-type which applies to discontinuous operators in non-compact domains. In order to do so, we present a modification of a recent Schauder-type theorem due to Pouso. We apply our result to second-order boundary value problems with discontinuous nonlinearities. We include an example to illustrate our theory.
\end{abstract}

\maketitle

\section{Introduction}
In the recent and interesting paper~\cite{pouso}, Pouso proved a novel version of Schauder's theorem for discontinuous operators in compact sets. Pouso used this tool to prove new results on the existence of solutions of a widely studied second order ordinary differential equation (ODE) subject to Dirichlet boundary conditions (BCs), namely
$$u''=f(t,u),\quad u(0)=u(1)=0,$$
where $f$ is a $L^1$-bounded nonlinearity. The approach in~\cite{pouso} relies on a careful use of ideas of set-valued analysis and viability theory. 

In this manuscript, we further develop the ideas of Pouso. Firstly we prove that a Schauder-type theorem for discontinuous operators can be formulated for arbitrary nonempty, closed and convex (not necessarily bounded) subsets of a Banach space. Secondly we apply our new result to prove the existence of solutions of a large class 
of discontinuous second order ODE subject to separated BCs, complementing the results of ~\cite{pouso} and improving them also in the special case of Dirichlet BCs.

\section{Schauder's fixed point theorem for discontinuous operators}

For the sake of completeness, we begin this Section by recalling the classical Schauder's fixed point theorem.

\begin{thm}\label{schauder}\cite[Theorem 2.A]{zeidler} Let $K$ be a nonempty, closed, bounded, convex subset of a Banach space $X$ and suppose that $T:K \longrightarrow K$ is a compact operator (that is, $T$ is continuous and maps bounded sets into precompact ones). Then $T$ has a fixed point.
\end{thm}

A well-known consequence of Theorem~\ref{schauder} is the following.

\begin{cor}\cite[Corollary 2.13]{zeidler}\label{corsh} Let $K$ be a nonempty, compact and convex subset of a Banach space $X$ and $T:K \longrightarrow K$ a continuous operator. Then $T$ has a fixed point.
\end{cor}

The main result in \cite{pouso} is an improvement of Corollary~\ref{corsh}, where the continuity of the operator $T$ is replaced by a weaker assumption. We briefly describe the main idea: given a compact subset $K$ of a Banach space $X$ and an operator $T:K \longrightarrow K$ that can be discontinuous, it is possible to
construct a multivalued mapping $\T$ by `convexifying' $T$ as follows:
\begin{equation}\label{T}
\T u:= \bigcap_{\varepsilon >0} \, \overline{\rm co} \, (T(B_{\varepsilon}(u) \cap K)) \ \mbox{ for every $u \in K$},
\end{equation}
where $B_{\varepsilon}(u)$ denotes the closed ball centered in $u$ and radius $\varepsilon$, and $\overline{\rm co}$ denotes the closed convex hull. The operator $\T$ in \eqref{T} is an upper semi-continuous mapping with convex and compact values (see \cite{aucel},\cite{deimling}), and therefore Kakutani's fixed point theorem guarantees that $\T$ has a fixed point in $K$. If we impose and extra assumption that, roughly speaking, states that a fixed point of $\T$ must be a fixed point of $T$, then we obtain the desired result.

The fo\-llo\-wing characterisation sheds light on the definition of the multivalued operator~$\T$. It is formulated for compact subsets, but it works for arbitrary nonempty subsets of a Banach space (see also \cite[Proposition 3.2]{pouso}).

\begin{pro} \label{char} Let $K$ be a compact subset of a Banach space $X$ and $T:K \longrightarrow K$. Then the following statements are equivalent:

\begin{enumerate}

\item $y \in \T u$, where $\T$ is as in (\ref{T});

\item for every $\varepsilon >0$ and every $\rho >0$ there exists a finite family of vectors $u_i \in B_{\varepsilon}(u) \cap K$ and coefficients $\lambda_i \in [0,1]$ ($i=1, \ldots,m$) such that $\sum \lambda_i =1$ and
$$
\left\| y - \sum_{i=1}^m \lambda_i Tu_i \right\| < \rho.$$

\end{enumerate}

\end{pro}

The variant of Schauder's theorem in compact subsets given by Pouso is the following.
\begin{thm}\cite[Theorem 3.1]{pouso}\label{pousoth} Let $K$ be a nonempty, compact and convex subset of a normed space $X$ and 
 $T: K \longrightarrow K$. Furthermore assume that
 $$
\{u\} \bigcap \T u \subset \{Tu\} \ \mbox{ for all $u \in K$},$$
where $\T$ is as in (\ref{T}).
Then $T$ has a fixed point.
\end{thm}
Theorem~\ref{pousoth} is very interesting and powerful; however, when one wants to look for solutions for a certain boundary value problem (BVP), the fact of working in a compact domain could be quite restrictive. In order to overcome this difficulty, firstly we recall that Theorem~\ref{schauder} admits the following extension to unbounded domains.
\begin{thm}\label{schauder2}\cite[Theorem 4.4.10]{lloyd} Let $M$ be nonempty, closed and convex subset of a Banach space $X$ and $T:M \longrightarrow M$ a continuous operator. If $T(M)$ is precompact then $T$ has a fixed point.
\end{thm}
Secondly, we recall the following result due to Bohnenlust-Karlin.
\begin{thm}\cite[Corollary 9.8]{zeidler}\label{BK}
Let $M$ be a nonempty, closed and convex subset of a Banach space $X$ and suppose that
\begin{enumerate}
\item[$(i)$] $T:M\to 2^M$ is upper semi-continuous;
\item[$(ii)$] $T(M)$ is relatively compact in $X$;
\item[$(iii)$] $T(u)$ is nonempty, closed and convex for all $u \in M$.
\end{enumerate}
Then $T$ has a fixed point.
\end{thm}

Now we introduce the main result in this Section, which is an extension of Theorem~\ref{schauder2} to the case of discontinuous operators.

\begin{thm}\label{sch} Let $M$ be a nonempty, closed and convex subset of a Banach space $X$ and $T:M \longrightarrow M$ a mapping satisfying

\begin{enumerate}

\item[$(i)$] $T(M)$ is relatively compact in $X$;

\item[$(ii)$] $\{u\} \cap \T u \subset \{Tu\}$ for all $u \in M$, where $\T$ is as in (\ref{T}).

\end{enumerate}

Then $T$ has a fixed point in $M$.

\end{thm}
\begin{proof}
The multivalued operator $\T$ is upper semi-continuous with $\T u$ nonempty, convex and compact for each $u \in M$. Now we show that condition $(i)$ implies that $\T (M)$ is relatively compact on $X$. Indeed, for each $u \in M$ and all $\varepsilon >0$ we have that $\overline{\rm co} \, T(B_{\varepsilon}(u) \cap M) \subset \overline{\rm co} \, T(M)$, and therefore $\T (M)$ is a closed subset of the compact set $\overline{\rm co} \, T(M)$ (note that the closed convex hull of a compact set in a Banach space is also compact, see for example \cite[Theorem 5.35]{alibor}).

Since $\T (M)$ is relatively compact, we obtain by application of Theorem~\ref{BK} that $\T$ has a fixed point. Finally, the condition $(ii)$ implies that the obtained fixed point of $\T$ is a fixed point of~$T$.
\end{proof}

\section{Second-order BVPs with separated BCs}

In this Section we apply the previous abstract result on fixed points for discontinuous operators in order to look for $W^{2,1}$-solutions for the following singular second-order ODE with separated BCs:
\begin{equation}\label{problem}
\left\{
\begin{array}{ll}
u''(t)+g(t)f(t,u(t))=0, \ \text{for almost every (a.e.)}\ t \in I=[0,1], \\
\alpha u(0) - \beta u'(0)=0, \\
\gamma u(1) + \delta u'(1)=0,
\end{array}
\right.
\end{equation}
where $\alpha, \beta, \gamma, \delta \ge 0$ and $\Gamma=\gamma \beta + \alpha \gamma + \alpha \delta >0.$

This kind of second order BVPs have received a lot of attention in the literature. For example, in the monograph \cite{coha} the method of lower and upper solutions is used to look for $\mathcal{C}^2$-solutions in the case of continuous nonlinearities and $W^{2,1}$-solutions in the case of Carath\'eodory ones. This method is also applied in \cite{carepo} to a continuous $\varphi$-Laplacian problem with separated BCs. On the other hand, a monotone method is applied in \cite{cahe} in order to look for extremal solutions for a functional problem with derivative dependence in the nonlinearity. As a main novelty of the present work, we allow the nonlinearity $f$ to have a countable number of discontinuities with respect to its spatial variable and we require no monotonicity conditions. Moreover, the linear part can be singular.

To apply our new fixed point theorem to the BVP~\eqref{problem}, we recall that a function $u \in W^{2,1}(I)$ is a solution of~\eqref{problem} if (and only if) $u$ is a solution of the following Hammerstein integral equation:
\begin{equation}\label{ham}
u(t)=\int_0^1 k(t,s) g(s) f(s,u(s)) \, ds,
\end{equation}
whenever the integral in (\ref{ham}) has sense and where $k$ is the corresponding Green's function, which is given by (see for example \cite{lan})
\begin{equation}\label{k}
k(t,s)=\dfrac{1}{\Gamma} \, \left\{
\begin{array}{ll}
(\gamma + \delta - \gamma t)(\beta + \alpha s), \quad \mbox{if $0 \le s \le t \le 1$}, \\
(\beta + \alpha t)(\gamma + \delta - \gamma s), \quad \mbox{if $0 \le t < s \le 1$}.
\end{array}
\right.
\end{equation}
It is known~\cite{lan} that $k$ is non-negative. Furthermore note that $k$ is continuous (and therefore bounded) in the square $[0,1] \times [0,1]$ and that its partial derivatives $\dfrac{\partial k}{\partial t}$ and $\dfrac{\partial k}{\partial s}$ can be discontinuous in the diagonal $t=s$. However, these partial derivatives are essentially bounded on the square.

In the sequel we consider the Banach space $X=\mathcal{C}^1(I)$ of continuously differentiable functions defined on $I$, with the norm
$$
||u||=\sup_{t \in I} |u(t)| + \sup_{t \in I} |u'(t)|.
$$
\begin{lem}\label{lemma1} Assume that:
\begin{enumerate}
\item[$(H_1)$] $g \in L^1(I)$;
\item[$(H_2)$] there exist $R >0$ and $H_R \in L^{\infty}(I)$ such that for a.e. $t \in I$ and all $u \in [-R,R]$ we have $|f(t,u)| \le H_R(t)$;
\item[$(H_3)$] the following estimate holds,
$$
||H_R||_{L^{\infty}} \, (M_1+M_2) \le R,$$
where
\begin{equation}\label{M}
M_1=\sup_{t \in I} \int_0^1 k(t,s) |g(s)| \, ds, \quad M_2= \sup_{t \in I} \int_0^1 \left| \dfrac{\partial k}{\partial t}(t,s) \right| g(s)| \, ds;\end{equation}
\item[$(H_4)$] for each $u \in \overline{B}_R=\{u \in X \, : \, ||u|| \le R\}$ the composition $t \in I \longmapsto f(t,u(t))$ is a measurable function.

\end{enumerate}

Then the operator $T: \overline{B}_R \longrightarrow X$ given by
$$
Tu(t)=\int_0^1 k(t,s) g(s) f(s,u(s)) \, ds$$
is well-defined and maps $\overline{B}_R$ into itself.

\end{lem}

\begin{proof}

Let $R>0$ given by condition $(H_2)$. First, note that the kernel $k$ has the form (\ref{k}), therefore for each $t \in [0,1]$ $k(t,\cdot)$ is a continuous function and for $s\neq t$ function $s \in [0,1] \to \dfrac{\partial k}{\partial t}(t,s)$ is well defined and integrable. Then, the conditions $(H_1)$, $(H_2)$ and $(H_4)$ imply that for $u \in \overline{B}_R$ the function $Tu$ is well defined.

On the other hand, if $u \in \overline{B}_R$ we have 
\begin{align*}
||Tu|| &\le \sup_{t \in I} \int_0^1 k(t,s) |g(s)| |f(s,u(s))| \, ds + \sup_{t \in I} \int_0^1 \left| \dfrac{\partial k}{\partial t}(t,s) \right| |g(s)| |f(s,u(s)| \, ds \\
& \le ||H_R||_{\infty} \, (M_1+M_2),
\end{align*}
and then condition $(H_3)$ implies that $||Tu|| \le R$.
\end{proof}
\begin{lem}\label{lemma2} Under the assumptions of Lemma~\ref{lemma1}, $T(\overline{B}_R)$ is relatively compact in $X$.
\end{lem}
\begin{proof}
We have shown in Lemma~\ref{lemma1} that $T(\overline{B}_R) \subset \overline{B}_R$, therefore the set $T(\overline{B}_R)$ is totally bounded in $X$. Now, to see that $T(\overline{B}_R)$ is equicontinuous we only have to notice that for a.e. $t \in I$ and every $u \in \overline{B}_R$ it is
$$
|(Tu)''(t)| \le |g(t)| \, H_R(t),$$
which implies that
$$
|(Tu)'(t)-(Tu)'(s)| \le \int_t^s |(Tu)''(r)| \, dr \le \int_t^s |g(r)| H_R(r) \, dr.$$
Then $T(\overline{B}_R)$ is relatively compact in $X$.
\end{proof}
In a similar way as in Definition 4.1 of~\cite{pouso}, we introduce the admissible discontinuities for our nonlinearities.
\begin{defn} We say that $\gamma:[a,b] \subset I \longrightarrow \R$, $\gamma \in W^{2,1}([a,b])$, is an admissible discontinuity curve for the differential equation $u''(t)+g(t)f(t,u(t))=0$ if one of the following conditions holds:
\begin{enumerate}
\item[(i)] $-\gamma''(t)=g(t) f(t,\gamma(t))$ for a.e. $t \in [a,b]$;
\item[(ii)] there exist $\psi \in L^1([a,b],[0,+\infty))$ and $\varepsilon >0$ such that
\begin{equation}\label{inv1} \mbox{ either $-\gamma''(t) + \psi(t) <g(t)f(t,y)$ for a.e. $t \in [a,b]$ and all $y \in [\gamma(t)-\varepsilon,\gamma(t)+\varepsilon]$}, \end{equation}
\begin{equation}\label{inv2} \mbox{ or $-\gamma''(t) - \psi(t) > g(t)f(t,y)$ for a.e. $t \in [a,b]$ and all $y \in [\gamma(t)-\varepsilon,\gamma(t)+\varepsilon]$.} \end{equation}
\end{enumerate}
If (i) holds then we say that $\gamma$ is \emph{viable} for the differential equation; if (ii) holds we say that $\gamma$ is \emph{inviable}.
\end{defn}
The previous definition says, roughly speaking, that a time-dependent discontinuity curve $\gamma$ is admissible if one of the following holds: either $\gamma$ solves 
the differential equation on its domain or, if it does not, the solutions are pushed `far away' from $\gamma$. 

The following is the main result in this Section.

\begin{thm}\label{thmapp}
 Let $f$ and $g$ satisfy $(H_1)-(H_4)$ and the following:

\begin{enumerate}

\item[$(H_5)$] there exist admissible discontinuity curves $\gamma_n: I_n =[a_n,b_n] \longrightarrow \R$, $n \in \N$, such that for a.e. $t \in I$ the function $f(t,\cdot)$ is continuous in $[-R,R] \setminus \bigcup_{n \, : \, t \in I_n} \{ \gamma_n(t)\}.$

\end{enumerate}

Then problem (\ref{problem}) has at least one solution in $\overline{B}_R$.

\end{thm}

\begin{proof}

We consider the multivalued operator $\T$ associated to $T$ as in (\ref{T}). Therefore $\T$ is upper semi-continuous with nonempty, convex and compact values and, as $T$, maps $\overline{B}_R$ into itself. Moreover, $\T(\overline{B}_R)$ is relatively compact in $X$ by Lemma~\ref{lemma2}. Therefore if we show that $\{u\} \cap \T u \subset \{Tu\}$ then we obtain, by Theorem~\ref{sch}, that $T$ has a fixed point in $\overline{B}_R$, which corresponds to a solution of the BVP~\eqref{problem}. This part of the proof now follows the line to the one of \cite[Theorem 4.4]{pouso}, but we include it for the sake of completeness and for highlighting the main differences between the two results. Thus we fix $u \in \overline{B}_R$ and consider three cases.\smallskip

\noindent \emph{Case 1}: $m(\{t \in I_n \, : \, u(t)=\gamma_n(t)\})=0$ for all $n \in \N$.\\ 
Then we have that $f(t,\cdot)$ is continuous for a.e. $t \in I$, and therefore if $u_k \to u$ in $\overline{B}_R$ we obtain $f(t,x_k(t)) \to f(t,x(t))$ for a.e. $t \in I$. This, joint with $(H_2)$ and $(H_3)$, imply that $Tu_k$ converges uniformly to $Tu$ in $X$. Then, $T$ is continuous at $u$ and therefore we obtain $\T u= \{Tu\}$.\smallskip 

\noindent \emph{Case 2}: there exists $n \in \N$ such that $\gamma_n$ is inviable and $m(\{t \in I_n \, : \, u(t)=\gamma_n(t)\})>0$. \\
Therefore, assume that $\gamma_n$ satisfies \eqref{inv2} (the other case is similar), let $\psi \in L^1(I)$ and $\varepsilon >0$ given by \eqref{inv2} and set
$$
J=\{t \in I_n \, : \, u(t)=\gamma_n(t)\}, \quad M(t)=g(t) \, H_R(t).$$

Then we repeat the proof done in \cite[Theorem 4.4]{pouso} by taking there $M(t)=g(t) \, H_R(t)$ and it follows that $u \notin \T u$.\smallskip

\noindent \emph{Case 3}: $m(\{t \in I_n \, : \, u(t)=\gamma_n(t)\})>0$ only for some of those $n \in \N$ such that $\gamma_n$ is viable. \\
Again, it suffices to follow the referred proof by replacing $f(t,x(t))$ by $g(t)f(t,u(t))$ to obtain that, in this case, $u \in \T u$ implies $u=Tu$. 

Then, we have that $\{u\} \cap \T u \subset \{Tu\}$ for all $u \in \overline{B}_R.$ By application of Theorem~\ref{sch} we obtain that $T$ has at least one fixed point in $\overline{B}_R$, which corresponds to a solution of  the BVP~\eqref{problem} in $\overline{B}_R.$
\end{proof}
\begin{rem}
Note that if $f(t,0)=0$, then $0$ is a solution of the BVP~\eqref{problem}, therefore when $f(t,0)\neq 0$ Theorem~\ref{thmapp} provides the existence of a nontrivial solution. In addition, if $g(t)f(t,u)\geq 0$, then we obtain the existence of a \emph{non-negative} solution with a non-trivial norm.
\end{rem}
\begin{rem}
The improvement with the respect to Theorem 4.4 of~\cite{pouso} relies not only on the fact that we can deal with a more general set of BCs but also on the fact that we do not require \emph{global} $L^1$ estimates on $f$, allowing a more general class of nonlinearities.
\end{rem}
Finally, we illustrate our results in an example.
\begin{exa} 
For $n \in \N$ we denote by $\phi(n)$ the function such that $\phi(1)=2$ and for $n \ge 2$ $\phi(n)$ counts the number of divisors of $n$. Thus defined, $\phi(n) \ge 2$ for all $n \in \N$, $\phi$ is not bounded and, as there are infinite prime numbers, $\liminf_{n \to \infty} \phi(n)=2$. Now we define the function
\begin{equation}\label{f}
(t,u) \in (0,1] \times \R \longmapsto \tilde{f}(t,u)=\phi^{\lambda}(n(t,u)), \ \lambda \in (0,1),
\end{equation}
where
$$
n(t,u):=
\left\{
\begin{array}{ll}
1, \ & \text{if}\ u \in (-\infty,-t), \\
\\
n, & \text{if}\ -\dfrac{t}{n} \le u < -\dfrac{t}{n+1}  \ \text{and}\ -t \le u <0, \\
\\
n, & \text{if}\ (n-1) \sqrt{t} \le u < n \sqrt{t} \ \text{and}\ u \ge 0.
\end{array}
\right.
$$

We are concerned with the following ODE
\begin{equation}\label{de}
u''(t)=\dfrac{\phi^{\lambda}(n(t,u))}{\sqrt{t}}, \mbox{ for a.e. $t \in I=[0,1]$},\end{equation}
coupled with separated BCs.

We claim that this problem has at least one solution. In order to show this, note that we can rewrite the ODE~ \eqref{de} in the form $u''(t)+g(t)f(t,u(t))=0$, where $g(t)=\dfrac{1}{\sqrt{t}}$ and $f=-\tilde{f}$, $\tilde{f}$ as in (\ref{f}). We now show that the functions $g$ and $f$ satisfy conditions $(H_1)-(H_5)$.

First, it is clear that $g \in L^1(I)$ and so $(H_1)$ holds. On the other hand, as $\phi(n) \le \max\{2,n\}$ for all $n \in \N$, we obtain that for each $n \in \N$ we have $u \in [-n,n] \Rightarrow |f(t,u)| \le \max\{2,n\}^{\lambda}$. Then, if we take $R \in \N$, $R \ge 2$, large enough such that $M_1+M_2 \le R^{1-\lambda}$ ($M_1,M_2$ as in (\ref{M})) we can guarantee that $(H_2)$ and $(H_3)$ hold.

To check $(H_4)$, note that for every continuous function $u$ we can write the composition $t \in I \longmapsto f(t,u(t))$ as
\begin{equation}\label{comp}
t \longmapsto f(t,u(t))=-\sum_{n=1}^{\infty} \phi^{\lambda}(n) \left(\chi_{I_n}(t)+\chi_{J_n}(t)\right)+\phi(1) \chi_{K}(t),\end{equation}
where $\chi$ denotes the characteristic function and $I_n$, $J_n$, $K$ are the following measurable sets:
$$
\left\{
\begin{array}{l}
I_n=u^{-1}([(n-1) \sqrt{t},n\sqrt{t}) \cap [0,+\infty)), \ n \in \N,\\
\\
J_n=u^{-1}\Bigl(\Bigl[\dfrac{-1}{n}t,\dfrac{-1}{n+1} t\Bigr) \cap [-t,0)\Bigr), \ n \in \N, \\
\\
K=u^{-1}((-\infty,-t)).
\end{array}
\right.$$
Then (\ref{comp}) is a measurable function and therefore condition $(H_4)$ is satisfied.

Finally we check the condition $(H_5)$. For a.a. $t \in I$ function $f(t,\cdot)$ has a countable number of discontinuities of the form $\gamma_k(t)=k\sqrt{t}$, and $\hat{\gamma}_k(t)=\dfrac{-1}{k+1}t$, $k \in K \subset \N$, but all these discontinuity curves are inviable for the differential equation. Indeed, notice that for $k \in K$ and $t \in I$ we have $-\gamma_k''(t)=\dfrac{k}{4t^{3/2}}>0$, $-\hat{\gamma}_k''(t)=0$ and
$$
g(t)f(t,\gamma_n(t)) \le -\dfrac{2^{\lambda}}{\sqrt{t}} \le -2^{\lambda},
$$
taking into account that $\phi(n) \ge 2$ for all $n \in \N$. Then, condition $(H_5)$ holds.

We can conclude that the differential equation (\ref{de}) coupled with separated BCs has at least one solution in $\overline{B}_R$ provided that $M_1+M_2 \le R^{1-\lambda}$. Note that the solution is non-trivial since the zero function does not satisfy the ODE.

In the special case of $\alpha=\beta=\gamma=\delta=1, \lambda=1/3$ we obtain (rounded to the third decimal place) $M_1+M_2=2,336$ and $R=4$.

\end{exa}

\section*{Acknowledgements}
The authors wish to thank Professor Rodrigo L\'opez Pouso for constructive comments.
R. Figueroa was partially supported by Xunta de Galicia, Conseller\'ia de Cultura, Educaci\'on e Ordenaci\'on Universitaria, through the project EM2014/032 ``Ecuaci\'ons diferenciais non lineares"; and by by Ministerio de Econom\'ia y Competitividad of Spain under Grant MTM2010-15314, cofinanced by the European Community fund FEDER. G.~Infante was partially supported by G.N.A.M.P.A. - INdAM (Italy). 
This paper was partially written during a visit of R. Figueroa to the 
Dipartimento di Matematica e Informatica, Universit\`{a} della
Calabria. R. Figueroa wants to thank all the people of this Dipartimento for their kind and warm hospitality.

\end{document}